\documentclass{amsart}
\usepackage[top=1.2in,bottom=1.2in, left=1.5 in, right= 1.5 in, includehead,includefoot]{geometry}

\usepackage{amsmath}
\usepackage{amssymb}
\usepackage{amsthm}
\usepackage[numbers,sort&compress]{natbib}
\usepackage{enumerate}
\usepackage{dsfont}
\usepackage{color}
\usepackage{url}

\newcommand{\R}{{\mathbb R}}
\newcommand{\N}{{\mathbb N}}
\newcommand{\Z}{{\mathbb Z}}
\newcommand{\U}{{\mathbb U}}

\newcommand{\M}{{\mathbb M}}

\newcommand{\B}{{\mathbb B}}

\newcommand{\F}{{\mathcal F}}

\newcommand{\esp}{\mathbb{E}}

\newcommand{\Var}{\operatorname{Var}}
\newcommand{\Cov}{\operatorname{Cov}}

\newcommand{\ind}{\mathds{1}}
\def\ip#1{\left\langle#1\right\rangle}
\def\sip#1{\langle#1\rangle}

\renewcommand{\P}{\mathbb{P}}

\newtheorem{theorem}{Theorem}
\newtheorem{proposition}{Proposition}
\newtheorem{lemma}{Lemma}

\theoremstyle{definition}

\newtheorem{remark}{Remark}

\newcommand{\comment}[1]{}
\def\limn{\lim_{n\to\infty}}
\def\limsupn{\limsup_{n\to\infty}}
\def\mfa{\text{ for all }}
\def\vv#1{{\boldsymbol #1}}
\def\vva{{\vv a}}

\def\vvs{{\vv s}}
\def\vvt{{\vv t}}
\def\vvdelta{{\vv\delta}}
\def\indd#1{{\ind}_{\{#1\}}}
\def\inddd#1{{\ind}_{\left\{#1\right\}}}

\def\mmas{\text{ as }}
\def\proba{\mathbb P}
\def\esp{\mathbb E}
\def\pp#1{\left(#1\right)}

\def\bb#1{\left[#1\right]}
\def\ccbb#1{\left\{#1\right\}}
\def\abs#1{\left|#1\right|}

\def\summ#1#2#3{\sum_{#1=#2}^{#3}}
\def\sif#1#2{\sum_{#1=#2}^\infty}
\def\prodd#1#2#3{\prod_{#1=#2}^{#3}}
\def\mmid{\;\middle\vert\;}
\def\wt#1{\widetilde{#1}}
\def\wb#1{\overline{#1}}
\def\topp#1{^{(#1)}}

\def\mand{\mbox{ and }}
\def\qmwith{\quad\mbox{ with }\quad}

\def\calS{{\mathcal S}}
\def\calM{{\mathcal M}}
\def\calN{{\mathcal N}}
\def\calY{{\mathcal Y}}

\def\eqfdd{\stackrel{f.d.d.}=}

\def\weakto{\Rightarrow}
\def\fddto{\stackrel{f.d.d.}\rightarrow}
\def\mmod{\ {\rm mod }\ }
\def\id{{\rm id}}
\newcommand{\eqnh}{\begin{eqnarray*}}
\newcommand{\eqne}{\end{eqnarray*}}
\newcommand{\eqnhn}{\begin{eqnarray}}
\newcommand{\eqnen}{\end{eqnarray}}
\newcommand{\equh}{\begin{equation}}
\newcommand{\eque}{\end{equation}}

\def\floor#1{\left\lfloor #1\right\rfloor}

\begin{document}
\title[Infinite urn schemes and sssi stable processes]{From infinite urn schemes \\to self-similar stable processes}

\author{Olivier Durieu}
\address{
Olivier Durieu\\
Institut Denis Poisson, UMR CNRS 7013,
Universit\'e de Tours, Parc de Grandmont, 37200 Tours, France.
}
\email{olivier.durieu@univ-tours.fr}

\author{Gennady Samorodnitsky}
\address{
Gennady Samorodnitsky\\
School of Operations Research and Information Engineering
and Department of Statistical Science \\
Cornell University\\
220 Rhodes Hall\\
Ithaca, NY 14853, USA.
}
\email{gs18@cornell.edu}

\author{Yizao Wang}
\address
{
Yizao Wang\\
Department of Mathematical Sciences\\
University of Cincinnati\\
2815 Commons Way\\
Cincinnati, OH, 45221-0025, USA.
}
\email{yizao.wang@uc.edu}

\date{\today}

\keywords{Infinite urn scheme, regular variation, functional central limit theorem, self-similar process, stable process}
\subjclass[2010]{Primary, 60F17; 
 Secondary,  60G18, 
 60G52} 

\begin{abstract}We investigate the randomized Karlin model with parameter $\beta\in(0,1)$, which is based on an infinite urn scheme. It has been shown before that when the randomization is bounded, the so-called odd-occupancy process scales to a fractional Brownian motion with Hurst index $\beta/2\in(0,1/2)$. We show here that when the randomization is heavy-tailed with index $\alpha\in(0,2)$, then the odd-occupancy process scales to a 
$(\beta/\alpha)$-self-similar symmetric $\alpha$-stable process with stationary increments. 
\end{abstract}
\maketitle

\section{Introduction and main results}

Consider the following  infinite urn scheme. Suppose there is an
infinite number of  urns  labeled by $\N = \{1,2,\dots\}$, all
initially empty. Balls are thrown into the urns randomly one after
another.  At each round, a ball  is thrown independently into the urn with label $k$ with probability $p_k$, with $\sum_{k\ge 1} p_k = 1$. 
This random sampling strategy dates back to at least the 60s \citep{bahadur60number,karlin67central}.
The urns may represent different species in a population of interest,
and in various applications an interesting question is to infer the
population frequencies $(p_k)_{k\ge1}$;  see \citep{gnedin07notes} and
references therein. This  urn scheme has  been extensively
investigated in the literature on  combinatorial stochastic processes as it induces 
the so-called {\em paintbox partition} of $\N$, an infinite
exchangeable random partition; see for example
\citep{pitman06combinatorial}.

Asymptotic results for many  statistics of this urn scheme, and in
particular of the random partition it induces, have been investigated
in the literature (e.g.~\citep{gnedin07notes} and references therein),
including in particular the so-called {\em odd-occupancy process}.  In
the sequel, let $Y_n$ denote the label of the urn that the $n$-th ball
falls into, so $\proba(Y_n = k) = p_k$. Then $(Y_n)_{n\in\N}$ are
i.i.d. random variables, and we let $Y_{n,k} :=\summ i1n\indd{Y_i=k}$ denote the number of balls in the urn with label $k$ after first $n$ rounds. The odd-occupancy process is then
defined as
\[
U^*_n=\sif k1\inddd{Y_{n,k}\ \rm odd}.
\]
This process counts the number of urns that contain an odd number of
balls  after the first $n$  rounds.
An interpretation of this process due to \citet{spitzer64principles}
is as follows. One may associate to each urn a lightbulb, and start
the sampling procedure with all lightbulbs off. Each time a ball falls
in an urn, the corresponding lightbulb changes its status (from off to
on or from on to off). The process $U_n^*$ then represents the total
number of lightbulbs that are on  after the first $n$ rounds. 

\citet{karlin67central} proposed and investigated the aforementioned
model under the following assumptions on $(p_k)_{k\ge 1}$: $p_k$ is
decreasing in $k$  and 
\equh\label{eq:RVp}
\max\{k\ge1\mid p_k\ge1/t\} = t^\beta L(t), \quad t\ge0, \mbox{ for some } \beta\in(0,1),
\eque
where $L$ is a slowly varying function at infinity. 
Among many results, Karlin proved that 
\[
\frac{U_n^*-\esp U_n^*}{\sigma_n}\weakto \calN(0,1)
\]
as $n\to\infty$ with $\sigma_n = 2^{\beta-1}(\Gamma(1-\beta) n^\beta L(n))^{1/2}$. 
Here and in the sequel $\weakto$ denotes weak convergence and
$\calN(0,1)$ stands for the  standard normal distribution. 

We are interested in the {\em randomized version} of the odd-occupancy
process, defined as 
\equh\label{eq:Un}
U_n=\sif k1\varepsilon_k\inddd{Y_{n,k}\ \rm odd},
\eque
where $(\varepsilon_k)_{k\in\N}$ are i.i.d.~symmetric random variables
independent of $(Y_n)_{n\in\N}$. 
This randomization  was recently introduced in
\citep{durieu16infinite}, and it was shown in Corollary 2.8 therein
that with $(\varepsilon_k)_{k\in\N}$ being a sequence of independent
Rademacher ($\pm1$-valued and symmetric)  random variables,  
 with the same normalization $\sigma_n$,
\equh\label{eq:DW16}
\pp{\frac{U_{\floor{nt}}}{\sigma_n}}_{t\in[0,1]}\Rightarrow  2^{-(\beta-1)/2}\pp{\B^{\beta/2}_t}_{t\in[0,1]} \eque
in $D([0,1])$, where $\B^H$ is a standard fractional Brownian motion with Hurst index $H\in(0,1)$, a centered Gaussian process with covariance function
\[
 \Cov(\B^H_s,\B^H_t) = \frac12\pp{s^{2H}+t^{2H}-|t-s|^{2H}}, \quad s,t\ge 0. 
\]
 
The following aspect of the randomization and the resulting functional
central limit theorem is  particularly interesting. For an arbitrary
symmetric distribution of $\varepsilon_1$, the randomized
odd-occupancy process $U_n$ is the partial-sum process for a 
stationary sequence, 
\equh\label{eq:stationary}
U_n = X_1+\cdots + X_n \qmwith X_i = -\varepsilon_{Y_i}(-1)^{Y_{i,Y_i}}, \quad i\in\N,n\in\N.
\eque
Therefore, the infinite urn scheme provides a specific way to generate
a stationary sequence of  random variables whose marginal distribution
is the given symmetric law, and whose  partial-sum process is 
 the randomized odd-occupancy process.
Moreover, at least for the Rademacher marginal distribution, this
stationary sequence $(X_n)_{n\ge1}$ exhibits, in view of the limiting
result \eqref{eq:DW16}, anomalous behavior. Namely, the normalization
$\sigma_n$  has an order of magnitude different  from the ``usual''
$\sqrt n$ normalization needed for partial sums of  i.i.d.~random
variables with the same marginal distribution. Such a behavior
indicates  {\em long-range  dependence} in the stationary sequence
\citep{samorodnitsky16stochastic,pipiras17long,beran13long}. The
long-range dependence, in this case, is due to the underlying random
partition. In particular, the covariance function of $X$ is determined
by the law of the random  partition. 
In general, when $\varepsilon_1$ is symmetric and in the domain of attraction of the normal distribution, it is expected that a fractional Brownian motion still arises in the limit with the same order of scaling as in \eqref{eq:DW16}.

In this paper, we are interested in the randomized odd-occupancy
process when $\varepsilon_1$ has a heavy-tailed
distribution. Specifically, we will assume that $\varepsilon_1$ has
infinite variance and, even more specifically, is in the domain of
attraction of a non-Gaussian stable law. The stationary-process
representation \eqref{eq:stationary} is, clearly, still
valid. However, in a functional central limit theorem one expects now
a symmetric stable process that is self-similar with stationary
increments (we abbreviate the latter two properties as the sssi
property). The only sssi Gaussian process is the fractional Brownian
motion (see e.g.~\citep{pipiras17long}); however there are many different sssi symmetric stable
processes, which often arise in   limit theorems for the  partial sums
of stationary sequences with long-range dependence 
\citep{samorodnitsky16stochastic,samorodnitsky94stable,pipiras17stable}. While
many sssi stable processes have been well investigated in the
literature, new processes in this family are still being discovered
\citep{owada15functional}. One way to classify different sssi stable
processes is via the flow representation, introduced by
\citet{rosinski95structure} and developed by
\citet{samorodnitsky05null}, 
of the corresponding  increment processes
which is, by necessity, stationary. Certain ergodic-theoretical
properties of these flows are invariants for each 
stationary stable process, and processes corresponding to different
types of flows, namely {\em  positive}, {\em conservative null} and
{\em  dissipative}, have drastically different properties. 

The main result of this paper is to show that for the Karlin model
with a heavy-tailed randomization, the scaling limit of the randomized
odd-occupancy process is a  new self-similar symmetric stable process
with stationary increments.  We continue to  
  assume that $(\varepsilon_k)$ are i.i.d. and symmetric. For
  simplicity,
   we will also assume that they are in the normal domain of
  attraction of a symmetric stable law
and  in particular, for some
    $\alpha\in(0,2)$, 
\equh\label{eq:RVepsilon}
\lim_{x\to\infty}\frac{\proba(|\varepsilon_1|>x)}{x^{-\alpha}} =
C_\varepsilon\in(0,\infty). 
\eque
We will define the limiting sssi symmetric $\alpha$-stable
(S$\alpha$S) process in terms of stochastic integrals with respect to
an S$\alpha$S random measure as follows. Let  $(\Omega',\F',\proba')$
be a probability space and $N'$  a standard Poisson
process defined on this space. Let $\calM_{\alpha,\beta}$ be a 
S$\alpha$S random measure on $\R_+\times\Omega'$ with control measure
$\beta r^{-\beta-1}dr\proba'(d\omega')$; we refer the reader to
\citep{samorodnitsky94stable} for detailed information on S$\alpha$S
random measures and stochastic integrals with respect to these
measures. The random measure $\calM_{\alpha,\beta}$ is itself defined
on a probability space $(\Omega,\F,\proba)$, the same probability
space on which the limiting process $\U^{\alpha,\beta}$ 
is defined as
\equh\label{eq:Uab}
\U^{\alpha,\beta}_t := \int_{\R_+\times\Omega'}\indd{N'(tr)(\omega')\ \rm odd}\calM_{\alpha,\beta}(dr,d\omega'),\quad t\ge0.
\eque
(See \eqref{eq:stable_fdd} below for the characteristic function of finite-dimensional distributions of $\U^{\alpha,\beta}$.)

The process $\U^{\alpha,\beta}$ is, to the best of our knowledge,  a
new class of sssi S$\alpha$S  processes, with self-similarity index
$\beta/\alpha$. In particular, we  shall show that its increment
process is driven by a positive flow  \citep{rosinski95structure,samorodnitsky05null}.
The following is the main result of the paper; we use the notation
$\stackrel{f.d.d.}\rightarrow$ for convergence in finite-dimensional distributions. 
\begin{theorem}\label{thm:1}
Under the assumptions~\eqref{eq:RVp} and~\eqref{eq:RVepsilon}, with $b_n = (n^\beta L(n))^{1/\alpha}$,
\[
\pp{\frac{U_{\floor{nt}}}{b_n}}_{t\in[0,1]}\stackrel{f.d.d.}\rightarrow \sigma_\varepsilon\pp{\U^{\alpha,\beta}_t}_{t\in[0,1]},
\]
where $\sigma_\varepsilon^\alpha = C_\varepsilon\int_0^\infty
x^{-\alpha}\sin x\, dx$.
If, in addition, $\alpha\in(0,1)$, then convergence in
distribution in the Skorohod $J_1$-topology on $D([0,1])$ also holds. 
\end{theorem}
\begin{remark}We assume the law of $\varepsilon_1$ has
    a power tail and is symmetric. This is assumed both for simplicity
    and because,  in the symmetric case, 
  $\{U_n\}_{n\in\N}$ can be represented as the partial-sum process of
  a stationary sequence of random variables, the marginal law of which
  is $\varepsilon_1$. This follows the original motivation when
  randomizing the Karlin model in \citep{durieu16infinite}.  These 
  assumptions may be relaxed. The assumption of power tails may be
  replaced by the assumption of regularly varying tails with exponent
  $-\alpha$. In this case the theorem remains clearly true after slight
  modifications: in particular  the normalization may acquire another
  slowly varying function. Another relaxation would involve dropping the
  assumption of symmetry, while still keeping  
  $\varepsilon_1$  in the domain of attraction of an $\alpha$-stable
  law. It is very likely that a version of the theorem will hold in
  such a general case; it will likely involve both a modified
  normalization and a centering. Furthermore, the limit may then be a
  nonsymmetric version of the  process $\U^{\alpha,\beta}$. In order
  to gain the maximum insight in the limiting procedure and the
  limiting process while not getting bogged down in technical details,
  we have chosen to work under the present, relatively
  straightforward, conditions. 
\end{remark}
We will prove this theorem by first conditioning on the urn sampling  
sequence 
$(Y_n)_{n\ge 1}$. It turns out that the characteristic function of
finite-dimensional distributions of $U$ can be expressed in terms of
certain statistics of that sequence. 
The same idea can be used in 
the case of a bounded $\varepsilon_1$  (although the proof in
\citep{durieu16infinite} was different and actually more involved as
the results are stronger; see also \citep{durieu19random} for the same
idea applied to a generalization of Karlin model). Indeed, in this case, the
limit process is Gaussian, so one addresses the convergence of the 
covariance function by essentially examining the joint
even/odd-occupancies {\em at two different time points}. In the
present case, the limit is a stable process and, hence, no longer
characterized by bivariate distributions. Therefore,  as an
intermediate step, we have to examine the joint even/odd-occupancies
{\em at multiple time points} $(n_1,\dots,n_d)$. For this purpose, we
investigate the following multiparameter process 
\[
M_{n_1,\dots,n_d}^{\delta_1,\dots,\delta_d} = \sum_{k\ge 1}\inddd{Y_{n_1,k} = \delta_1\mmod 2, \dots, Y_{n_d,k} = \delta_d \mmod 2}, 
\]
with $n_1,\dots,n_d\in\N_0=\{0\}\cup\N$ and
$(\delta_1,\dots,\delta_d)\in\{0,1\}^d\setminus\{(0,\dots,0)\}$. We
refer to this process as the {\em multiparameter even/odd-occupancy
  process}. A weak law of large numbers for the process $M$ will turn
out to be sufficient 
to prove the first part of Theorem \ref{thm:1}. However, we will
establish a functional central limit theorem for the multiparameter
even/odd-occupancy   process 
(Theorem \ref{thm:M} below). The limit in that result can be viewed as a
multiparameter generalization of the bi-fractional Brownian motion
\citep{houdre03example} and, hence,  is of   interest on its own. 
There is a huge literature on limit theorems for various counting statistics of the Karlin model and its extensions (see e.g.~\citep{gnedin10limit,gnedin07notes,alsmeyer17functional} and references therein). However, the investigation of multiparameter even/odd-occupancy process above seems to be new. 

The paper is organized as follows.  Section \ref{sec:sssi} reviews the background on flow representation of stationary stable processes and shows that the increment process of $\U^{\alpha,\beta}$ is driven by a positive flow. Section \ref{sec:M} establishes limit theorems for the multiparameter even/odd-occupancy process $M$. Section \ref{sec:proof} presents the proof of main result.

\section{A new class of self-similar stable processes with stationary increments}\label{sec:sssi}
We start by verifying the self-similarity of the process $\U^{\alpha,\beta}$ introduced in \eqref{eq:Uab}.  It will also follow from Theorem \ref{thm:1} and the Lamperti theorem (see e.g.~\citep{samorodnitsky16stochastic}), but a direct argument is simple. 
\begin{proposition}The process $\U^{\alpha,\beta}$ is $(\beta/\alpha)$-self-similar. That is,
\[
\pp{\U^{\alpha,\beta}_{\lambda t}}_{t\ge0} \eqfdd \lambda^{\beta/\alpha}\pp{\U^{\alpha,\beta}_t}_{t\ge 0} \quad\mfa \lambda>0.
\]
\end{proposition}
\begin{proof}
Fix $\lambda>0$. For any $d\in\N, t_1,\dots,t_d\ge0,
a_1,\dots,a_d\in\R$,
 \begin{align}
\esp\exp\pp{i\summ k1d a_k\U^{\alpha,\beta}_{\lambda t_k}} & = \exp\pp{-\int_0^\infty\int_{\Omega'}\abs{\summ k1d a_k\inddd{N'(\lambda t_kr)(\omega')\ \rm odd}}^\alpha \proba'(d\omega')\beta r^{-\beta-1}dr}\label{eq:stable_fdd}
\\\nonumber
& = \exp\pp{-\int_0^\infty\int_{\Omega'}\lambda^\beta\abs{\summ k1d a_k\inddd{N'(t_kr)(\omega')\ \rm odd}}^\alpha \proba'(d\omega')\beta r^{-\beta-1}dr}\\
& =  \esp\exp\pp{-i\summ k1d
  a_k\lambda^{\beta/\alpha}\U^{\alpha,\beta}_{t_k}}, \nonumber
\end{align}
as required. 
\end{proof}

We now consider the increment process of $\U^{\alpha,\beta}$. We will
see that the increment process is stationary (which, of course, will
follow from  Theorem \ref{thm:1} as well). More importantly, we will
classify the flow structure of this process in the spirit of
\citet{rosinski95structure} and \citet{samorodnitsky05null}. 
We start with a background on the flow structure of stationary
S$\alpha$S processes. It was shown by 
\citet{rosinski95structure} that, given a stationary S$\alpha$S process
$X = (X_t)_{t\in\R}$ there exist a measurable space $(S,\calS,\mu)$,  a non-singular
flow 
$(T_t)_{t\in\R}$ on it, and a function $f\in L^\alpha(S,\mu)$,
such that  
\equh\label{eq:spectral}
(X_t)_{t\in\R}\stackrel{f.d.d.}=\pp{\int_S c_t(s)f\circ T_t(s)\pp{\frac{d\mu\circ T_t}{d\mu}(s)}^{1/\alpha}\calM_\alpha(ds)}_{t\in\R},
\eque
where $\calM_\alpha$ is an S$\alpha$S random measure on $(S,\calS)$
with control measure $\mu$, and 
$(c_t)_{t\in\R}$
 is a $\pm 1$-valued 
cocycle with respect to 
$(T_t)_{t\in\R}$. Recall that a non-singular flow
$(T_t)_{t\in \R}$ is a group of measurable maps from $S$ onto $S$ such
that  for all $t\in\R$ the measure $\mu\circ T$ is equivalent to
$\mu$. 
A cocycle 
$(c_t)_{t\in\R}$ is a family of measurable functions on
$(S,\calS)$ such that  for all $t_1,t_2\in\R$, $c_{t_1+t_2}(s) =
c_{t_1}(s)c_{t_2}\circ T_{t_1}(s)$ $\mu$-almost everywhere. See  
\citet{krengel85ergodic} and \citet{aaronson97introduction} for more
information. In particular,  $T$ 
induces a unique decomposition of $(S,\calS)$ modulo $\mu$, 
\[
S = P \cup CN \cup D,
\]
where $P, CN, D$ are disjoint $T$-invariant measurable subsets of $S$ and,
restricted to each subset (if non-empty), $T$ is positive,
conservative null and dissipative, respectively. This decomposition
generates a unique in law decomposition of the process $X$ in
\eqref{eq:spectral} into a sum of 3 independent stationary S$\alpha$S 
processes, $X=X^P + X^{CN} +X^D$, where $X^P$ corresponds to a
positive flow, $X^{CN}$ corresponds to a conservative null flow, and 
$X^D$ corresponds to a dissipative flow, with one or two of the
components, possibly, vanishing.  In fact, one can define all 3
processes as in \eqref{eq:spectral}, but integrating over $P, CN, D$
correspondingly. 
  It is known that $X^P$ is non-ergodic, $X^D$ is mixing, while 
$ X^{CN}$ is ergodic, and can be either mixing or non-mixing; see
\citep{rosinski95structure,samorodnitsky05null,samorodnitsky16stochastic}. 
The processes generated by a dissipative flow have necessarily a mixed
moving-average representation.  The least understood family of
processes are those generated by a conservative null flow. 

We will show that the increment process of $\U^{\alpha,\beta}$ is
generated by a positive flow. In fact, a stationary S$\alpha$S process
$X$ is generated by positive flow if and only if it can be represented as 
in \eqref{eq:spectral}, where now the control measure $\mu$ is a
probability measure invariant under action of the operators
$(T_t)_{t\in\R}$ (so that the factor $(d\mu\circ T_t/d\mu)^{1/\alpha}$ disappears); see   \citep[Remark 2.6]{samorodnitsky05null}. 

For this purpose, we first present a natural extension of $\U^{\alpha,\beta}$ to a stochastic process indexed by $t\in\R$. 
We may and will assume that $\Omega'$ is the space of
Radon measures on $\R$ equipped with the Borel $\sigma$-field
corresponding to the topology of vague convergence, 
$\proba'$ 
is the law of
the unit rate Poisson point process on $\R$  and
\[
N'(t)(\omega'):=\begin{cases}
\omega'([0,t]) & t\geq 0\\
\omega'([t,0)) & t<0.
\end{cases}
\]
In this way, we now define
\[
\U^{\alpha,\beta}_t := \int_{\R_+\times\Omega'}\indd{N'(tr)(\omega')\ \rm odd}\calM_{\alpha,\beta}(dr,d\omega'),\ t\in\R. 
\]
This definition extends \eqref{eq:Uab}.
\begin{proposition}
The  increment process of $\U^{\alpha,\beta}$ defined as
\[
X_t := \U^{\alpha,\beta}(t+1) - \U^{\alpha,\beta}(t),\  
t\in\R,
\]
is stationary and driven by a positive flow.
\end{proposition}
\begin{proof}
Let $m_\beta$ denote the measure $\beta r^{-\beta-1} dr$ on $\R_+$.
It follows from the stochastic integral representation of $\U^{\alpha,\beta}$ that,
\equh
\label{eq:Xfdd}
(X_t)_{t\in\R}   \eqfdd  \pp{\int_{\R_+\times\Omega'}\bigl( \indd{N'((t+1)r)(\omega')\ \rm odd} - \indd{N'(tr)(\omega')\ \rm odd}\bigr)\calM_{\alpha,\beta}(dr, d\omega')}_{t\in\R},
\eque
where $\calM_{\alpha,\beta}$  is the S$\alpha$S random measure
described above.

Let $\theta_t$ be the
standard left shift on the space of Radon measures on $\R$, $t\in\R$,
and define a group of measurable operators on $\R_+\times\Omega'$ by 
\[
T_t(r,\omega') := (r,\theta_{tr}(\omega')), \ t\in\R\,.
\]
If $\nu$ is a probability measure on $\R_+$ equivalent to $m_\beta$
then the probability measure $\mu=\nu\times \proba'$ on $\R_+\times\Omega'$ 
is preserved by the operators $(T_t)_{t\in\R}$. 
Denote $h=dm_\beta/d\nu$ and 
define 
$$
f(r,\omega')= h(r)^{1/\alpha}\indd{N'(r)(\omega')\ \rm odd}, 
$$
and
\[
c_t(r,\omega')=
\indd{N'(tr)(\omega')\ \rm even} - \indd{N'(tr)(\omega')\ \rm odd}, \
(r,\omega')\in \R_+\times\Omega'.
\]
 If $\calM_\alpha$ is 
an S$\alpha$S
random measure on $\R_+\times\Omega'$ with control measure $\mu$,
then, in law, \eqref{eq:Xfdd} is the same as 
$$
(X_t)_{t\in\R}   \eqfdd  \pp{\int_{\R_+\times\Omega'} c_t(r,\omega')
  f\circ T_t(r,\omega') \calM_{\alpha}(dr, d\omega')}_{t\in\R}.
$$
Since 
$(c_t)_{t\in\R}$ is, clearly, a $\pm 1$-valued cocycle, this will
establish 
both stationarity of the increment process and the fact that it is
driven by a positive flow once we check that the function $f\in
L^\alpha(\mu)$. However,  
\eqnh
\int|f|^\alpha d\mu & = & \int\indd{N'(r)(\omega')\ \rm odd}m_\beta(dr)\proba'(d\omega')\\
& = & \int_0^\infty\beta r^{-\beta-1}\proba'\pp{N'(r)\ {\rm odd}}dr \\
& = & \int_0^\infty \beta r^{-\beta-1}\frac12\pp{1-e^{-2r}}dr = \Gamma(1-\beta)2^{\beta-1}<\infty.
\eqne
\end{proof}

\section{The multiparameter even/odd-occupancy process}
\label{sec:M}
Throughout,   for $d\in\N$, $\vvt = (t_1,\dots,t_d)\in[0,1]^d$, we
write 
\[
\floor{n\vvt} = (\floor{nt_1},\dots,\floor{nt_d}) = (n_1,\dots,n_d) 
\]
and denote 
\[
\Lambda_d = \{0,1\}^d\setminus\{(0,\dots,0)\}.
\]
Let $\vvdelta\in \Lambda_d$ and consider the multiparameter even/odd-occupancy process
\[
M_{\floor{n\vvt}}^\vvdelta := \sif k1 \prodd j1d\inddd{Y_{n_j,k} =
  \delta_j\mmod 2},\quad n\in\N, \, \vvt\in[0,1]^d. 
\]
Let $\M^\vvdelta = (\M^\vvdelta_\vvt)_{\vvt\in[0,1]^d}$ be a centered
Gaussian random field with covariance function 
\[
\Cov(\M_\vvt^\vvdelta,\M_\vvs^\vvdelta) 
= \int_0^\infty \Cov\pp{\inddd{\vec N(r\vvt)=\vvdelta\mmod2},\inddd{\vec N(r\vvs)=\vvdelta\mmod2}}\beta r^{-\beta-1}dr,
\]
where $N$ is a standard Poisson process on $\R_+$ and 
\equh\label{eq:vvN}
\ccbb{\vec N(n\vvt) = \vvdelta \mmod 2} \equiv \ccbb{N(nt_j) =
  \delta_j\mmod 2 \mfa j=1,\dots,d}\,. 
\eque
The next result is  a limit theorem for
$M_{\floor{n\vvt}}^\vvdelta$. It uses the normalization 
 $d_n  = b_n^\alpha = n^\beta L(n)$, where $b_n$ is as in Theorem
 \ref{thm:1}. We use the new notation to emphasize the fact that the
 normalization in Theorem \ref{thm:M} does not depend on $\alpha$.
\begin{theorem}\label{thm:M}
Under the assumption~\eqref{eq:RVp},
for all $d\in\N, \, \vvt\in[0,1]^d, \, \vvdelta\in\Lambda_d$, 
\[
\lim_{n\to\infty}\frac{M_{\floor{n\vvt}}^\vvdelta}{d_n} =m_\vvt^\vvdelta
\]
in probability, where 
\[
m_\vvt^\vvdelta := \limn \frac{\esp M_{\floor {n\vvt}}^\vvdelta}{d_n} = \int_0^\infty\proba\pp{\vec N(r\vvt)=\vvdelta\mmod 2}\beta r^{-\beta-1}dr\,.
\]
Moreover, 
\[
\pp{\frac{M_{\floor{n\vvt}}^\vvdelta - \esp M_{\floor{n\vvt}}^\vvdelta}{\sqrt{d_n}}}_{\vvt\in[0,1]^d}\weakto \pp{\M_\vvt^\vvdelta}_{\vvt\in[0,1]^d}
\]
in $D([0,1]^d)$ with respect to the 
Skorohod topology
and the limiting random field has a
version with continuous sample paths. 
\end{theorem}
\begin{remark} For the weak convergence in Theorem \ref{thm:M} (and in Proposition \ref{prop:MM} below), we shall prove that it holds with respect to any topology generated by a complete separable metric on $D([0,1]^d)$  weaker than the uniform metric. For general background on weak convergence in $D([0,1]^d)$, see for example \citep{neuhaus71weak,straf72weak}.
\end{remark}
Only the first part of Theorem \ref{thm:M} is needed for 
Theorem \ref{thm:1}. However, the Gaussian random field
$\M_\vvt^\vvdelta$ is of interest on its own. In fact, it was shown in
\citep[Theorem~2.3]{durieu16infinite} that, 
when $d=1$, the process $M_n^1$ (which is simply the odd-occupancy
process) satisfies the weak convergence in Theorem \ref{thm:M} with the
limiting process  $\mathbb M^1$ being, up to a
multiplicative constant, the bi-fractional Brownian
motion \citep{lei09decomposition,houdre03example} with parameters $H =
1/2, K = \beta$. This is a 
centered Gaussian process with covariance function 
\[
\Cov(\M^1 _t,\M^1_s) = \Gamma(1-\beta)2^{\beta-2}\pp{(s+t)^\beta-
  |s-t|^\beta},\quad s,t\geq 0.
\]
Therefore, the limit obtained in Theorem \ref{thm:M} can be viewed as a random field generalization of the bi-fractional
 Brownian motion.  

In order to analyze the multiparameter even/odd-occupancy process
$M_{\floor{n\vvt}}^\vvdelta$, we introduce a Poissonization of the
underlying urn sampling sequence $(Y_n)_{n\in\N}$. Let $N$ be 
a standard Poisson process independent of $(Y_n)_{n\in\N}$ and
$(\varepsilon_n)_{n\in\N}$. Set  
\[
N_k(t):= \summ i1{N(t)}\indd{Y_i=k},\quad k\in\N, \, t\ge 0.
\]
Clearly $(N_k)_{k\in\N}$ are independent Poisson processes with
respective parameters  $(p_k)_{k\in\N}$. We use the notation
$\{\vec N_k(\vvt) = \vvdelta \mmod 2\}$ whose meaning is analogous to 
\eqref{eq:vvN}, and the Poissonized version of the multiparameter
even/odd-occupancy process is 
\[
\wt M_\vvt^\vvdelta := \sum_{k=1}^\infty\inddd{\vec
  N_k(\vvt)=\vvdelta\mmod2},\quad \vvt\in\R_+^d.
\]

\begin{lemma}\label{lem:M}
Under the assumption~\eqref{eq:RVp}, for all $d\in\N, \,
\vvt\in[0,1]^d, \, \vvdelta\in\Lambda_d$,  
\[
\lim_{n\to\infty}\frac{\wt M_{n\vvt}^\vvdelta}{d_n} =
m_\vvt^\vvdelta\quad\text{ in }\, L^2.
\]
\end{lemma}
\begin{proof}
Consider a Radon measure $\nu$ on $\R_+$ defined by $\nu :=
\sif k1 \delta_{1/p_k}$. By the assumption~\eqref{eq:RVp}, we have
$\nu(x) := \nu([0,x]) = x^\beta L(x)$, for all $x>0$. 
Observe that
\begin{align*}
\esp \wt M_{n\vvt}^\vvdelta& = \sum_{k=1}^\infty\esp\inddd{\vec N_k(n\vvt)=\vvdelta\mmod 2} = \sif k1\proba\pp{\vec N(np_k\vvt)=\vvdelta\mmod2}\\
& =\int_0^\infty\proba\pp{\vec N(n\vvt/x) =\vvdelta\mmod 2}\nu(dx) = \int_0^\infty\varphi_\vvt\pp{\frac nx}\nu(dx)
\end{align*}
with 
$\varphi_\vvt(s):=\proba(\vec N(s\vvt) = \vvdelta\mmod 2)$.
Assume, without loss of generality, that $t_1<\cdots<t_d$ and $\delta_1 = 1$.
We have the following explicit expression for $\varphi_\vvt$,
\begin{align*}
 \varphi_\vvt(s)
 &=\P(N(st_1)\ {\rm odd})\prod_{k=2}^d\P(N(s(t_k-t_{k-1}))=|\delta_k-\delta_{k-1}|\mmod 2)\\
 &=\frac{1-e^{-2st_1}}{2^d}\prod_{k=2}^d \bb{1+(-1)^{|\delta_k-\delta_{k-1}|}e^{-2s(t_k-t_{k+1})}},
\end{align*}
from which we can easily deduce that for $\vvt$ fixed,
$\varphi_\vvt'(s)$ is bounded and there exist $T>0$ and $C>0$ such
that $|\varphi'_\vvt(s)| \le Ce^{-sT}$ for all $s>0$.  
Integrating by parts gives us 
\[
\int_0^\infty\varphi_\vvt\pp{\frac nx}\nu(dx) = 
\int_0^\infty \frac
1{x^2}\varphi_\vvt'\pp{\frac 1x}\nu(nx)dx. 
\]
A standard
argument using the Potter bounds (\citep[Corollary
10.5.8]{samorodnitsky16stochastic}) tells us that 
\begin{align*}
\lim_{n\to\infty}\frac{1}{\nu(n)}\int_0^\infty\varphi_\vvt\pp{\frac
  nx}\nu(dx) & 
= \int_0^\infty \frac1{x^2}\varphi_\vvt'\pp{\frac 1x}\lim_{n\to\infty}\frac{\nu(nx)}{\nu(n)}dx
\\
& =\int_0^\infty\varphi_\vvt'\pp{\frac 1x}x^{\beta-2} dx  = \int_0^\infty\varphi_\vvt\pp{\frac1x}\beta x^{\beta-1}dx 
\\ &
 = \int_0^\infty\proba\pp{\vec N\pp{\vvt r}=\vvdelta\mmod2}\beta r^{-\beta-1}dr.
\end{align*}
Since 
\[
\Var(\wt M_{n\vvt}^\vvdelta) = \sif k1 \Var\pp{\inddd{\vec N_k(n\vvt)=\vvdelta\mmod2}} \leq \sif k1 \proba\pp{\vec N_k(n\vvt)=\vvdelta\mmod2}
=\esp\wt M_{n\vvt}^\vvdelta, 
\]
the  $L^2$ 
convergence follows. 
\end{proof}

\begin{proposition}\label{prop:MM}
Under the assumption~\eqref{eq:RVp}, for all $d\in\N, \,
\vvdelta\in\Lambda_d$,
\[
\pp{\frac{\wt M_{n\vvt}^\vvdelta - \esp\wt M_{n\vvt}^\vvdelta}{\sqrt{d_n}}}_{\vvt\in[0,1]^d}\weakto \pp{\M_\vvt^\vvdelta}_{\vvt\in[0,1]^d}
\]
in $D([0,1]^d)$ with 
respect to the 
Skorohod topology, and the limiting random field has a
version with continuous sample paths. 
\end{proposition}
\begin{proof}
We continue to use the notation in the proof of  Lemma \ref{lem:M}.
For $\vvs,\vvt\in[0,1]^d$, 
\begin{align*}
\Cov(\wt M_{n\vvt}^\vvdelta,\wt M_{n\vvs}^\vvdelta)  & = \sif k1\Cov\pp{\inddd{\vec N_k(n\vvt)=\vvdelta\mmod2},\inddd{\vec N_k(n\vvs)=\vvdelta\mmod2}}\\
& = \int_0^\infty \Cov\pp{\inddd{\vec N(n\vvt/x)=\vvdelta\mmod2},\inddd{\vec N(n\vvs/x)=\vvdelta\mmod2}}\nu(dx).
\end{align*}
Since $d_n=\nu(n)$, we can use the same argument 
 as in the proof of Lemma \ref{lem:M} to show that 
\begin{align*}
  \limn\frac{\Cov(\wt M_{n\vvt}^\vvdelta,\wt M_{n\vvs}^\vvdelta)}{d_n} & = \int_0^\infty \Cov\pp{\indd{\vec N(r\vvt)=\vvdelta\mmod2},\indd{\vec N(r\vvs)=\vvdelta\mmod2}}\beta r^{-\beta-1}dr\\ & = \Cov(\M_\vvt^\vvdelta,\M_\vvs^\vvdelta).
\end{align*}
Write
\[
\wb M_{n\vvt}^\vvdelta = \wt M_{n\vvt}^\vvdelta - \esp \wt M_{n\vvt}^\vvdelta.
\]
For each $n$, $\wb M_{n\vvt}^\vvdelta$ is the sum of independent
random variables that are centered and uniformly bounded (by
2). Further, $\sqrt{d_n}\to\infty$ as $n\to\infty$. Therefore, the
Lindeberg--Feller condition holds. Together with 
 the Cram\'er--Wold's device, this shows the convergence of the
 finite-dimensional distributions in the statement of the
 proposition. 

It remains to prove the tightness. Letting 
\[
A_\eta:=\ccbb{(\vvt,\vvt')\in[0,1]^d\times[0,1]^d:\max_{j=1,\dots,d}|t_j-t_j'|\leq\eta}, 
\]
it is enough to  prove that for all $\epsilon>0$,
\equh\label{eq:tightnessM}
\lim_{\eta\downarrow0}\limsupn\proba\pp{\sup_{(\vvt,\vvt')\in A_\eta}\frac{|\wb M_{n\vvt}^\vvdelta - \wb M_{n\vvt'}^\vvdelta|}{\sqrt{d_n}}>\epsilon} = 0.
\eque
It is easy to see that it suffices to show~\eqref{eq:tightnessM} with $A_\eta$ replaced by 
\[
A_\eta\topp i:=\ccbb{(\vvt,\vvt')\in[0,1]^d\times[0,1]^d:0\le t_i\le
  t_i'\le t_i+\eta, \, t_j = t_j', \, j\neq i},
\]
for all $i=1,\dots,d$, and we prove \eqref{eq:tightnessM} for
$A_\eta\topp1$.
By \citep[Lemma 3.7]{durieu16infinite} which is due to a chaining argument, it suffices to establish the following: 
for all $(\vvt,\vvt')\in A_\eta\topp1$, 
\equh\label{eq:lem3}
\abs{\wb M_{n\vvt}^\vvdelta - \wb M_{n\vvt'}^\vvdelta}\leq N(n(t_1+\eta))-N(nt_1)+ n\eta \quad\mbox{ almost surely,}
\eque
and
for all $p\in\N$ and $\gamma\in(0,\beta)$, there exists $C_{p,\gamma}>0$ such that for all $(\vvt,\vvt')\in A_\eta\topp1$, 
\equh\label{eq:lem2}
\esp\abs{\wb M_{n\vvt}^\vvdelta-\wb M_{n\vvt'}^\vvdelta}^{2p}\leq C_{p,\gamma}\pp{|t_1-t_1'|^{\gamma p}\nu(n)^{p} + |t_1-t_1'|^\gamma \nu(n)}.
\eque
Fix $(\vvt,\vvt')\in A_\eta\topp1$ and to simplify the notation introduce $B_k = \{N_k(n\vvt) = \vvdelta \mmod 2\}$ and $B_k' = \{N_k(n\vvt') = \vvdelta \mmod2\}$. We have
\begin{align*}
B_k\Delta B_k' \subset \ccbb{N_k(nt_1') - N_k(nt_1)\ne 0}.
\end{align*}
To show~\eqref{eq:lem3}, it suffices to observe that
\begin{align*}
\abs{\wb M^\vvdelta_{n\vvt} -\wb M^\vvdelta_{n\vvt'}}
& = \abs{\sif k1\bigl( \ind_{B_k} - \ind_{B_k'} - \proba(B_k) +
  \proba(B_k')\bigr) } \leq\sif k1 \ind_{B_k\Delta B_k'}+\sif k1 \proba(B_k\Delta B_k')\\
& \leq N(nt_1')-N(nt_1) + \esp N(n(t_1'-t_1)) \le N(nt_1')-N(nt_1)+n\eta.
\end{align*}
To show~\eqref{eq:lem2}, 
by Rosenthal's inequality \citep{rosenthal70subspaces}, for some constant $C$ depending on $p\in\N$ only, 
\begin{align*}
&\esp\abs{\wb M_{n\vvt}^\vvdelta - \wb M_{n\vvt}^\vvdelta}^{2p} \\
&\le 
C\bb{\sif k1 \esp\abs{\ind_{B_k}-\ind_{B_k'}-\proba(B_k)+\proba(B_k')}^{2p} + \pp{\sif k1 \esp \abs{\ind_{B_k}-\ind_{B_k'}-\proba(B_k)+\proba(B_k')}^2}^p}.
\end{align*}
Since
$$
\esp|\ind_A-\ind_B-\proba(A)+\proba(B)|^{2p} \le
2^{2p-1}\Var(\ind_A-\ind_B) \leq 2^{2p-1}\proba(A\Delta B)\,,
$$
the above expression does not exceed 
\[C\bb{\sif k1 \proba\pp{B_k\Delta B_k'} + \pp{\sif k1 \proba\pp{B_k \Delta B_k'}}^p} \le C\bb{V(n(t_1'-t_1)) +V(n(t_1'-t_1))^p}
\]
with
\[
V(t) = \sif k1 \proba(N_k(t)\ne 0) = \sif k1\pp{1-e^{-p_kt}}.
\]
By  \citep[Lemma 3.1]{durieu16infinite}, for each
$\gamma\in(0,\beta)$, there exists $C_\gamma$ such that  
\[
V(nt)\le C_\gamma t^\gamma \nu(n) \mbox{ for all } t\in[0,1], \, n\in\N. 
\]
Therefore, \eqref{eq:lem2} follows. 
Therefore, we have proved \eqref{eq:tightnessM}, which also implies that the limit process has a version with continuous sample path \citep[Theorem 5.6]{straf72weak}. 
This completes the proof.
\end{proof}
\begin{proof}[Proof of Theorem \ref{thm:M}]
We will prove the second part of the theorem using a multivariate 
version of the change-of-time lemma from
\citep[p.~151]{billingsley99convergence} (the proof of which is the same as that
of the univariate version). Since the limiting random field is
continuous, the first part will then follow. 

Let $(\tau_n)_{n\in\N_0}$ be the arrival times of the Poisson process
$N$ (with $\tau_0 = 0$). For $\vvt\in \R_+^d$, set
$ \vv\tau_{\floor{n\vvt}} = (\tau_{\floor{nt_1}},\dots,\tau_{\floor{nt_d}})\in\R_+^d$. 
By the strong law of large numbers and monotonicity, 
 \[
 \lambda_n\equiv \pp{\frac{\vv\tau_{\floor {n\vvt}}}{n}\wedge\vv2}_{\vvt\in[0,2]^d}\to \vv{\id}\quad \mbox{ almost surely,}
 \]
as $n\to\infty$, where $\vv\id$ is the identity function from $[0,2]^d$ to $[0,2]^d$.
By the multivariate change-of-time lemma, 
 \[
\pp{\frac{\wb M_{n((\vv\tau_{\floor{n\vvt}}/n)\wedge \vv2)}^\vvdelta}{\sqrt{d_n}}}_{\vvt\in[0,2]^d}\weakto\pp{\M_\vvt^{\vvdelta}}_{\vvt\in[0,2]^d}.
 \]
In particular, the convergence also holds if the random fields are
restricted to $\vvt\in[0,1]^d$. Since 
 \[
 \pp{\frac{M^\vvdelta_{\floor{n\vvt}} - \esp M_{\floor{n\vvt}}^\vvdelta}{\sqrt{d_n}}}_{\vvt\in[0,1]^d} = \pp{\frac{\wb M_{n (\vv\tau_{\floor{n\vvt}}/n)}^\vvdelta}{\sqrt{d_n}}}_{\vvt\in[0,1]^d},
 \]
and 
 \[
\limn \proba\pp{\pp{\frac{\vv\tau_{\floor{n\vvt}}}n \wedge
    \vv2}_{\vvt\in[0,1]^d} =
  \pp{\frac{\vv\tau_{\floor{n\vvt}}}n}_{\vvt\in[0,1]^d}} = 1, 
 \]
the desired result follows.  
\end{proof}

\section{Proof of Theorem \ref{thm:1}}\label{sec:proof}

\begin{proof}[Proof of convergence of finite-dimensional
  distributions]
Let $d\geq 1$ and fix $\vvt=(t_1,\ldots, t_d)\in [0,1]^d$ and
$\vva=(a_1,\ldots,a_d)\in \R^d$.  Note that  
\begin{align}\label{eq:limitfdd}
\esp\exp\pp{i\summ j1da_j\U^{\alpha,\beta}_{t_j}}
& = \exp\pp{-\int_{\R_+\times\Omega'}\abs{\summ
  j1da_j\indd{N'(r{t_j})\ \rm odd}}^\alpha m_\beta(dr)\, \proba'(d\omega')}\\
& =
  \exp\pp{-\int_{\R_+\times\Omega'}\sum_{\vvdelta\in\Lambda_d}\abs{\ip{\vva,\vvdelta}}^\alpha\inddd{\vec
  N'(r\vvt) = \vvdelta\mmod 2}\beta r^{-\beta-1}dr\, d\proba'}\nonumber\\
&= \exp\pp{-\sum_{\vvdelta\in\Lambda_d}\abs{\ip{\vva,\vvdelta}}^\alpha m_\vvt^\vvdelta}.\nonumber
\end{align}
Similarly, with  the notation $n_j = \floor{nt_j}$,
\begin{align*}
\esp\exp\pp{i\summ j1d a_j\frac{U_{n_j}}{b_n}} & = \esp\exp\pp{i\summ j1d\frac{a_j}{b_n}\sif k1\varepsilon_k\inddd{Y_{n_j,k}\ \rm odd}}\nonumber\\
& = \esp\exp\pp{i\sum_{\vv\delta\in\Lambda_d}\frac{\sip{\vv a,\vv\delta}}{b_n}\sif k1\varepsilon_k\prodd j1d\indd{Y_{n_j,k}=\delta_j\mmod 2}}.
\end{align*}
Let $\calY$ denote the $\sigma$-algebra generated by $(Y_i)_{i\in\N}$.  Then, the above expression becomes
\[
\esp\ccbb{\prod_{\vv\delta\in\Lambda_d}\esp\bb{\exp\pp{i\frac{\ip{\vv a,\vv \delta}}{b_n}\summ \ell1{M_{\floor{n\vvt}}^{\vv\delta}}\varepsilon_\ell}\mmid\calY}} = \esp\pp{\prod_{\vv\delta\in\Lambda_d}\phi\pp{\frac{\ip{\vv a,\vvdelta}}{b_n}}^{M_{\floor{n\vvt}}^\vvdelta}},
\]
with $\phi(\theta) = \esp\exp(i\theta\varepsilon_1)$ being the characteristic function of $\varepsilon_1$.
Therefore,
\begin{align}
\esp\exp \pp{i\summ j1da_j\frac{U_{n_j}}{b_n}} & = \esp\pp{\prod_{\vvdelta\in\Lambda_d}\phi\pp{\frac{\ip{\vva,\vvdelta}}{b_n}}^{M_{\floor{n\vvt}}^\vvdelta}} \nonumber = \esp\exp\pp{\sum_{\vvdelta\in\Lambda_d}M_{\floor{n\vvt}}^\vvdelta\log\phi\pp{\frac{\ip{\vva,\vvdelta}}{b_n}}} \nonumber\\
& = \esp\exp\pp{-\sum_{\vvdelta\in\Lambda_d}\frac{M_{\floor{n\vvt}}^\vvdelta}{b_n^\alpha} \sigma_{\varepsilon}^\alpha\abs{\ip{\vv a,\vvdelta}}^\alpha \frac{\log\phi(c_n)}{-\sigma_{\varepsilon}^\alpha| c_n|^\alpha}},\label{eq:chfM}
\end{align}
with $c_n:=\ip{\vv a,\vvdelta}/b_n$. Recall that assumption~\eqref{eq:RVepsilon}
implies that (see e.g.~\citep[Theorem 8.1.10]{bingham87regular})
\[
\log\phi(\theta)\sim \phi(\theta)-1\sim -\sigma_{\varepsilon}^\alpha|\theta|^\alpha \mbox{ as } \theta\to0.
\]
By Theorem \ref{thm:M}  and  dominated convergence theorem the 
expression in \eqref{eq:chfM} converges to the characteristic function
in \eqref{eq:limitfdd}. 
\end{proof}
In the remainder of this section we consider the case $\alpha\in(0,1)$
and prove the tightness of the
sequence of processes $(U_{\floor{nt}}/b_n)_{t\in[0,1]}$, $n\ge 1$ in
the $J_1$-topology on the space $D([0,1])$. For $\kappa>0$ we
decompose   $ U_n=U^{+,\kappa}_n+ U_n^{-,\kappa}$   with
\begin{align*}
U_n^{+,\kappa}& =\sum_{k\ge 1} \varepsilon_k\indd{|\varepsilon_k|>\kappa b_n}\indd{Y_{n,k}\text{ odd}}\\
 U_n^{-,\kappa}& =\sum_{k\ge 1} \varepsilon_k\indd{|\varepsilon_k|\le\kappa b_n}\indd{Y_{n,k}\text{ odd}}.
\end{align*}

We start by showing that for any $\alpha\in(0,2)$ and any $\kappa>0$,
the sequence of the laws of the processes
$(U^{+,\kappa}_{\floor{nt}}/{b_n})_{t\in[0,1]}$, $n\ge 1$, is tight in
the $J_1$-topology. To this end define 
\[
\mathcal T\topp\kappa(n)  = \ccbb{i\in\{1,\dots,n\}:Y_i = k \
  \text{for $k$ such that $|\varepsilon_k|>\kappa b_n$}}.
\]
\begin{lemma}
For any $\kappa>0$,
\[
 \lim_{\delta\downarrow 0}\limsup_{n\to\infty} \P\left(\exists\,
   i_1,i_2 \in \mathcal T^{(\kappa)}(n) \mbox{ such that }
   |i_1-i_2|<n\delta\right) =0. 
\]
\end{lemma}
\begin{proof}
Fix $\kappa>0$ and let $\epsilon>0$. 
Let $\Theta_{n}=\sum_{k\ge 1}\indd{|\varepsilon_k|>\kappa b_n}p_k$,
$n\geq 1$. We first prove that
\begin{equation}\label{eq:Theta}
\exists\; C=C(\epsilon,\kappa)<+\infty \mbox{ such that }\limsup_{n\to\infty}\P(\Theta_{n}>C/n)\le\epsilon.
\end{equation}
To see this, choose $c>0$. Recalling the notation $\nu(x) = x^\beta L(x)$ 
we consider $\Theta_n^{(c)}=\sum_{k>\nu(c n)}\indd{|\varepsilon_k|>\kappa
  b_n}p_k$. Note that  by \eqref{eq:RVepsilon}, 
\begin{align*}
 \P(\Theta_n^{(c)}\ne \Theta_n)
 &\le  \P(\exists\, k\le \nu(cn) \mbox{ such that }|\varepsilon_k|>\kappa b_n)\\
 &=1- (1-\P(|\varepsilon_1|>\kappa b_n))^{[\nu(cn)]}\to 1-e^{-c^\beta
   C_\varepsilon \kappa^{-\alpha}}, \mmas n\to\infty, 
\end{align*}
and so we can choose $c=c(\epsilon, \kappa)$ such that 
\begin{equation}\label{eq:Theta^c}
\limsup_{n\to\infty} \P(\Theta_n^{(c)}\ne \Theta_n)\le \epsilon/2.
\end{equation}
Further, 
\[
\sum_{k> \nu(cn)} p_k = \sum_{k\ge 1} p_k \indd{p_k< 1/cn}=\int_{(cn,\infty)} \frac1x \nu(dx)\sim \frac{\beta}{1-\beta}(cn)^{\beta-1} L(n), 
 \mmas n\to\infty,
\]
where the equivalence is due to integration by parts and an
application of a Karamata Theorem (see \citep[Theorem 1
p.~281]{feller71introduction}). 
Thus,
\[
  \esp \Theta_n^{(c)}\sim \frac{\beta}{1-\beta}(cn)^{\beta-1} L(n)\P(|\varepsilon_1|>\kappa b_n)
  \sim  \frac{\beta}{1-\beta} c^{\beta-1} C_\varepsilon\kappa^{-\alpha} n^{-1}, \mmas n\to\infty.
\]
Now \eqref{eq:Theta} follows from the Markov inequality and
\eqref{eq:Theta^c}.  By \eqref{eq:Theta},
\begin{multline*}
\limsup_{n\to\infty} \proba\left(\exists\, i_1, i_2 \in \mathcal T^{(\kappa)}(n) \mbox{ such that } |i_1-i_2|<n\delta\right) \\
\le \limsup_{n\to\infty} \proba\pp{\exists\, i_1, i_2 \in \mathcal T^{(\kappa)}(n) \mbox{ such that } |i_1-i_2|<n\delta \mmid \Theta_n\le C/n} + \epsilon.
\end{multline*}
Letting $(B_i\topp n)_{i\in\N}$ be i.i.d.~Bernoulli random variables
with parameter $C/n$, we can bound the 
 first   term above by 
\[
\limsupn \proba\pp{\exists\, i_1,i_2\in\{1,\dots,n\} \mbox{ such that } |i_1-i_2|<n\delta \mand B_{i_1}\topp n = B_{i_2}\topp n = 1}.
\]
The latter probability has a limit, equal to 
  the probability that a Poisson point process with intensity $C$ over
  $[0,1]$ has two points  less than $\delta$ apart. 
This probability goes to zero as $\delta\downarrow 0$. This completes the proof.
\end{proof}

\begin{proposition}\label{prop:+tight}
 For all $\alpha\in(0,2), \kappa>0$, the sequence of processes
 $(U^{+,\kappa}(\floor{nt})/{b_n})_{t\in[0,1]}$, $n\ge 1$, is tight in
 the $J_1$-topology on $D([0,1])$. 
\end{proposition}
\begin{proof}
Since $\sup_{s\in[r,t]}|U^{+,\kappa}_{\floor{nr}}-U^{+,\kappa}_{\floor{ns}}|=0$ as soon as $\mathcal T^{(\kappa)}(\floor{nt})\setminus\mathcal T^{(\kappa)}(\floor{nr})=\emptyset$,
from the preceding lemma we infer that for all $\eta>0$,
\[
 \lim_{\delta\downarrow0}\limsup_{n\to\infty}\P\left(\sup_{\substack{0\le r\le s \le t\le 1\\ |r-t|<\delta}}\abs{U^{+,\kappa}_{\floor{nr}}-U^{+,\kappa}_{\floor{ns}}}\wedge\abs{U^{+,\kappa}_{\floor{ns}}-U^{+,\kappa}_{\floor{nt}}}>\eta\right)=0,
\] 
which yields the tightness of $(U^{+,\kappa}_{\floor{n\cdot}})_{n\ge1}$ (see \citep{billingsley99convergence}).
\end{proof}

\begin{proof}[Proof of tightness of $\bigl((U_{\floor{nt}}/b_n)_{t\in[0,1]}\bigr)$  when
  $\alpha\in(0,1)$] 
Let $\alpha\in(0,1)$.
 In view of  Proposition~\ref{prop:+tight},  it is sufficient to show that for any $\eta>0$,
\begin{equation}\label{eq:-tight}
 \lim_{\kappa\to0}\limsup_{n\to\infty} \P\left(\sup_{t\in[0,1]} \left|\frac{U^{-,\kappa}_{\floor{nt}}}{b_n}\right|>\eta \right)=0.
\end{equation}
Note that 
\begin{align*}
\P\left(\sup_{t\in[0,1]}
  \left|\frac{U^{-,\kappa}_{\floor{nt}}}{b_n}\right|>\eta \right)& \le 
\P\left( \sum_{k\ge
  1}\left|\frac{\varepsilon_k}{b_n}\right|\indd{|\varepsilon_k|\le\kappa
  b_n}\indd{Y_{n,k}>0} >\eta \right) \\
& \le
  \P\left(\sum_{k=1}^{K_n}\left|\frac{\varepsilon_k}{b_n}\right|\indd{|\varepsilon_k|\le\kappa
  b_n}>\eta\right), 
\end{align*}
where $K_n$ is the number of nonempty boxes at time $n$ in the
infinite urn scheme. Since for large $n$,
\begin{align*}
\esp\left( \sum_{k=1}^{K_n}\left|\frac{\varepsilon_k}{b_n}\right|\indd{|\varepsilon_k|\le\kappa  b_n}\right) 
&= \esp(K_n) \esp\left(\left|\frac{\varepsilon_1}{b_n}\right|\indd{|\varepsilon_1|\le\kappa b_n}\right) \\
& \leq Cb_n^\alpha b_n^{-1}\kappa b_n \P\left( |\varepsilon_1|>\kappa
  b_n\right) \to CC_\varepsilon\kappa^{1-\alpha}
\end{align*}
(for a finite constant $C$) by \citep[Proposition~2]{gnedin07notes}
and Karamata's theorem, \eqref{eq:-tight} follows by Markov's
inequality. 
 \end{proof}
\begin{remark}
Whether or not the full weak convergence in Theorem \ref{thm:1} holds
when $\alpha\in[1,2)$ remains an open question. In this case, it is
not even clear to us whether $\U^{\alpha,\beta}$ has a c\`adl\`ag
modification: sufficient conditions are given, for example, in 
\citep[Theorem 4.3]{basseoconnor13uniform},  but they are not
satisfied here. 
\end{remark}

\section{Discussions}
There are a few limit theorems for other statistics in \citep{durieu16infinite} that we have not addressed yet. We provide a brief discussions here focusing on other processes that appear in the limit.  As for the proofs, they do not require new ideas (if one ignores the tightness issues). 

For the odd-occupancy process $U_n$ in \eqref{eq:Un}, one can write, for $\beta<\alpha$, 
\begin{align*}
U_n&  = \sif k1\varepsilon_k\inddd{Y_{n,k}\ \rm odd} = \sif k1\varepsilon_k\pp{\inddd{Y_{n,k}\ \rm odd} - \proba(Y_{n,k}\ \rm odd)} + \sif k1\varepsilon_k\proba(Y_{n,k}\ \rm odd)\\ 
& =: U_n\topp1+U_n\topp2,
\end{align*}
and one could eventually prove that
\equh\label{eq:decomposition}
\frac1{b_n}\pp{U_{\floor{nt}}, U_{\floor{nt}}\topp1,U_{\floor{nt}}\topp2}_{t\in[0,1]} \fddto\sigma_\varepsilon\pp{\U_t^{\alpha,\beta}, \U_t^{\alpha,\beta,(1)}, \U_t^{\alpha,\beta,(2)}}
\eque
as $n\to\infty$, with
\begin{align*}
\U_t^{\alpha,\beta,(1)} & = \int_{\R_+\times\Omega'}\bb{\inddd{N'(tr)(\omega')\ \rm odd}-\proba'(N'(tr)\ \rm odd)}\calM_{\alpha,\beta}(dr,d\omega')\\
\U_t^{\alpha,\beta,(2)} & = \int_{\R_+\times\Omega'}\proba'(N'(tr)\ \rm odd)\calM_{\alpha,\beta}(dr,d\omega'),
\end{align*}
where here and below $\calM_{\alpha,\beta}$ and $N'$ are as before. 
We need the constraint $\beta<\alpha$ so that $U\topp1_n$, $U\topp 2_n$, $\U_t^{\alpha,\beta,(1)}$ and $\U_t^{\alpha,\beta,(2)}$  are well defined. Such a weak convergence, for the original randomized Karlin model ($\varepsilon_k\in\{\pm1\}$ and $\alpha = 2$), has been proved in \citep{durieu16infinite}. An appealing feature is that the corresponding decomposition of 
\[
\U_t^{2,\beta} = \U_t^{2,\beta,(1)} + \U_t^{2,\beta,(2)}
\]
recovers a decomposition of fractional Brownian motion by a bi-fractional Brownian motion and another smooth self-similar Gaussian process discovered in \citep{lei09decomposition}, and in particular, in this case the two processes are independent. For $\alpha\in(0,2)$, 
the convergence of finite-dimensional distributions to the decomposition still holds, although $\U^{\alpha,\beta,(1)}$ and $\U^{\alpha,\beta,(2)}$ are no longer independent.  The convergence in \eqref{eq:decomposition} could be established by computing characteristic functions and applying the same conditioning trick. 

Another statistics considered in \citep{karlin67central,durieu16infinite} is the {\em occupancy process}
\[
Z_n := \sif k1 \varepsilon_k\inddd{Y_{n,k}>0}.
\]
Correspondingly, the limit process is 
\[
\Z^{\alpha,\beta}_t = \int_{\R_+\times\Omega'}\inddd{N'(tr)>0}\calM_{\alpha,\beta}(dr,d\omega'), \quad t\ge 0.
\]
At the same time, this is nothing but a time-changed S$\alpha$S L\'evy process, as one can verify by computing the characteristic functions that 
\[
\pp{\Z^{\alpha,\beta}_t}_{t\ge 0} \eqfdd\pp{\Z^\alpha(t^\beta)}_{t\ge 0},
\]
where $(\Z^\alpha(t))_{t\ge 0}$ is an S$\alpha$S L\'evy process ($\esp e^{i\theta\Z^\alpha(1)} = e^{-|\theta|^\alpha}, \theta\in\R$). 
A similar decomposition for $\Z^{\alpha,\beta}$, and the corresponding limit theorem as in \eqref{eq:decomposition} can also be established, again by computing characteristic functions.  
The corresponding results for the Gaussian case ($\alpha = 2$) have already been investigated in \citep[Theorem 2.1]{durieu16infinite}.

\subsection*{Acknowledgments} 
The first author would like to thank the hospitality and financial support from Taft Research Center and Department of Mathematical Sciences at University of Cincinnati, for his visits in 2016 and 2017.
The second author's research was partially supported by NSF grant
  DMS-1506783 and the ARO grant  W911NF-12-10385 at Cornell
  University.  The third author's  research was partially supported by 
the NSA grants H98230-14-1-0318 and H98230-16-1-0322, the ARO grant W911NF-17-1-0006, and Charles Phelps Taft Research Center 
at University of Cincinnati.

\bibliographystyle{apalike}
\bibliography{references}
\end{document}